\documentclass[a4paper]{amsart}
\usepackage{amscd}
\usepackage{amsthm}
\usepackage{graphicx}
\usepackage{amsmath}
\usepackage{amsfonts}
\usepackage{amssymb}%
\providecommand{\U}[1]{\protect\rule{.1in}{.1in}}

\newtheorem{theorem}{Theorem}

\newtheorem{corollary}[theorem]{Corollary}

\newtheorem{lemma}[theorem]{Lemma}

\newtheorem{proposition}[theorem]{Proposition}
\newtheorem{remark}[theorem]{Remark}

\newcommand\gr{\mathop{\rm gr}}

\begin{document}

\title[$L^{p}$-improving properties of certain singular measures on the Heisenberg group]{$L^{p}$-improving properties of certain singular measures on the Heisenberg group}
\author{Pablo Rocha}
\address{Departamento de Matem\'atica, Universidad Nacional del Sur, Av. Alem 1253 - Bah\'{\i}a Blanca 8000, Buenos
Aires, Argentina.}
\email{pablo.rocha@uns.edu.ar}
\thanks{\textbf{Key words and phrases}: Heisenberg group; singular Borel measure; $L^{p}$-improving property.}
\thanks{\textbf{2.020 Math. Subject Classification}: 43A80, 42A38.}

\begin{abstract} 
Let $\mu_A$ be the singular measure on the Heisenberg group $\mathbb{H}^{n}$ supported on the graph of the quadratic function 
$\varphi(y) = y^{t}Ay$, where $A$ is a $2n \times 2n$ real symmetric matrix. If $\det (2A \pm J) \neq 0$, we prove that the operator of convolution by $\mu_A$ on the right is bounded from $L^{\frac{2n+2}{2n+1}}(\mathbb{H}^{n})$ to $L^{2n+2}(\mathbb{H}^{n})$. We also study the type set  of the measures
$d\nu_{\gamma}(y,s) = \eta(y) |y|^{-\gamma} d\mu_{A}(y,s)$, for $0 \leq \gamma < 2n$, where $\eta$ is a cut-off function around the origin on 
$\mathbb{R}^{2n}$. Moreover, for $\gamma =0$ we characterize the type set of $\nu_{0}$.
\end{abstract}
\maketitle

\section{Introduction}

Let $I_n$ be the $n \times n$ identity matrix and $J$ be the $2n \times 2n$ skew-symmetric matrix given by
\begin{equation} \label{matrix J}
J= \left( \begin{array}{cc}
                           0 & I_n \\
                          -I_n & 0 \\
                                      \end{array} \right).
\end{equation}
\\ The Heisenberg group is $\mathbb{H}^{n} = \mathbb{R}^{2n} \times \mathbb{R}$ endowed with the group law (non-commutative)
\[
\left( x,t\right) \cdot \left(y,s\right)
=\left( x+y, \, t+s+\left\langle x,y\right\rangle \right),
\] 
where $\langle x,y \rangle$ is the standard symplectic form on $\mathbb{R}^{2n}$, i.e. $\langle x,y \rangle = x^{t} J y$,
with neutral element $(0,0)$, and with inverse $\left( x,t\right) ^{-1}=\left(-x,-t\right) $. The topology in $\mathbb{H}^{n}$ is the induced by $\mathbb{R}^{2n+1}$, so the borelian sets of $\mathbb{H}^{n}$
are identified with those of $\mathbb{R}^{2n+1}$. The Haar measure in $\mathbb{H}^{n}$ is the Lebesgue measure of $\mathbb{R}^{2n+1}$, thus $L^{p}(\mathbb{H}^{n}) \equiv L^{p}(\mathbb{R}^{2n+1})$.
Given a borelian function $f : \mathbb{H}^{n} \to \mathbb{C}$ and a Borel measure $\mu$ on $\mathbb{H}^{n}$, define the convolution by $\mu$ on the right by
\begin{equation}
(f \ast \mu) (x,t) = \int_{\mathbb{H}^{n}} f((x,t) \cdot (y,s)^{-1}) \, d\mu(y,s), \label{conv}
\end{equation}
provided the integral exists.

A Borel measure $\mu$ on the Heisenberg group $\mathbb{H}^{n}$ is said to be $L^{p}$-\textit{improving} if the operator
$T_{\mu} : f \mapsto f \ast \mu$ is bounded from $L^{p}(\mathbb{H}^{n})$ into $L^{q}(\mathbb{H}^{n})$ for some $1 \leq p < q < \infty$.

A remarkable fact is that singular measures can be $L^{p}$-improving. If in (\ref{conv}) we replace the Heisenberg group $\mathbb{H}^{n}$ by
$\mathbb{R}^{n}$ with the ordinary convolution in $\mathbb{R}^{n}$ and considering there $\mu = \eta \sigma_{M}$, where $\sigma_{M}$ is the surface measure on a given manifold $M$ (in $\mathbb{R}^{n}$) and $\eta$ is a smooth cut-off function, then the $L^{p}$-improving properties of a measure of this type are closely related to the existence of a certain amount of curvature of the manifold $M$ 
(see \cite{litt}, \cite{ober}, \cite{pan}). A similar result holds on general Lie groups (see Theorem 1.1, pp. 362 in \cite{ricci2}).

A more delicate problem consists in determining the exact range of pairs $(p,q)$ for which $L^{p} \ast \mu \subseteq L^{q}$ embeds continuously.
Given a manifold $M$ (in $\mathbb{H}^{n}$), define the \textit{type set} $E_{\eta \sigma_{M}}$ by
\[
E_{\eta \sigma_{M}} =\left\{ \left(\frac{1}{p}, \frac{1}{q} \right) \in [0,1] \times [0,1] : \| T_{\eta \sigma_{M}} \|_{p,q} < \infty \right\}.
\]
A very interesting survey of results concerning the type sets for convolution operators with singular measures in $\mathbb{R}^{n}$ can be found in \cite{ricci}.

In the $\mathbb{H}^{n}$ setting, S. Secco in \cite{secco} and \cite{secco2} obtained $L^{p}$-\textit{improving} properties of measures
supported on curves in $\mathbb{H}^{1}$, under certain assumptions. In \cite{ricci2}, F. Ricci and E. Stein showed that the type set of the
measure given by (\ref{nu}), for the case $\varphi \equiv 0$, $\gamma =0$ and $n=1$, is the triangle with vertices
$(0,0),$ $(1,1),$ and $\left( \frac{3}{4},\frac{1}{4}\right)$. In \cite{G-R} and \cite{G-R2}, the author jointly with T. Godoy generalized the work of Ricci and Stein for the case $\varphi(w) = w^{t}Aw = \sum_{j=1}^{n} \alpha_j |w_j|^{2}$, where $A$ is a $2n \times 2n$ real diagonal matrix such that $a_{ii} = a_{(i+1) (i+1)}$ for $i=2j-1$ with $j =1, 2, ..., n$, $\alpha_j = a_{(2j-1)(2j-1)}$, $w_j \in \mathbb{R}^{2}$, $0 \leq \gamma < 2n$, and $n \in \mathbb{N}$. There we also gave some examples of surfaces with degenerate curvature at the origin.

Let $\varphi : \mathbb{R}^{2n} \to \mathbb{R}$ be the function defined by $\varphi(y) = y^{t}Ay$, where $A$ is a $2n \times 2n$ real symmetric matrix. It is well known that if $A$ is an arbitrary matrix, then there exists a symmetric matrix $\widetilde{A}$ such that $y^{t}A y = y^{t}\widetilde{A} y$, for all $y$. We consider two borelian measures on $\mathbb{H}^{n}$ supported on the graph of $\varphi$ , $\mu_A$ and $\nu_{\gamma}$, $0 \leq \gamma < 2n$, given by
\begin{equation*}
\mu_A(E) = \int_{\mathbb{R}^{2n}} \, \chi_{E}(y, \varphi(y)) \, dy, 
\end{equation*}
and
\begin{equation}
\nu_{\gamma}(E) = \int_{\mathbb{R}^{2n}} \, \chi_{E}(y, \varphi(y)) \, \eta(y) \, |y|^{-\gamma} \, dy, \label{nu}
\end{equation}
where $\eta: \mathbb{R}^{2n} \to [0,1]$ is a smooth cut-off function such that $\eta(y) = 1$ if $|y| \leq 1$, $\eta(y) = 0$ if $|y| \geq 2$, and
$E$ is a borelian set of $\mathbb{H}^{n}$.
Let $T_{\mu_A}f = f \ast \mu_A$ and $T_{\nu_{\gamma}}f = f \ast \nu_{\gamma}$ be the operators of convolution by $\mu_A$ and 
$\nu_{\gamma}$ on the right respectively.

We are interested in studying the $L^{p}$-improving properties of the operator $T_{\mu_A}$ and in the characterization of the 
\textit{type set} $E_{\nu_{\gamma}}$. We point out that our measure $\mu_A$ is not the surface measure on the graph $\gr (\varphi)$ of $\varphi$, however the measures $\eta \mu_A$ and $\eta \sigma_{\gr (\varphi)}$ are equivalent, see Proposition 2 below, so $E_{\eta \mu_A} = 
E_{\eta \sigma_{\gr (\varphi)}}$. 

The following restrictions for the type sets $E_{\nu_{\gamma}}$, $0 \leq \gamma < 2n$, were proved in \cite{G-R} and \cite{G-R2} for the case 
$\varphi(w_1, ..., w_n) = \sum_{j=1}^{n} \alpha_j |w_j|^{2}$ with $w_j \in \mathbb{R}^{2}$. It is easy to see that such an argument works as well for our function $\varphi(y) =y^{t}Ay$.

Thus if $\left( \frac{1}{p},\frac{1}{q}\right) \in E_{\nu_{\gamma}}$, $0 \leq \gamma < 2n$, then
\begin{equation}
p\leq q, \,\,\,\,\,\,\,\,\,\,\,\, \frac{1}{q}\geq \frac{2n+1}{p}-2n, \,\,\,\,\,\,\,\,\,\,\,\,  \frac{1}{q}\geq\frac{1}{(2n+1)p}. 
\label{restriccion 1}
\end{equation}

Another necessary condition for the pair $(\frac{1}{p}, \frac{1}{q})$ to be in $E_{\nu_{\gamma}}$ is the following:
\begin{equation}
\frac{1}{q}\geq \frac{1}{p}-\frac{2n-\gamma }{2n+2}. \label{restriccion 2}
\end{equation}
This last condition is relevant only for the case $0 < \gamma < 2n$.

Let $D$ be the point of intersection, in the $(\frac{1}{p}, \frac{1}{q})$ plane, of the lines $\frac{1}{q}=\frac{2n+1}{p}
-2n$, $\frac{1}{q}=\frac{1}{p}-\frac{2n-\gamma }{2n+2}$, and let $D^{\prime }$ be its symmetric image with respect to the symmetry axis 
$\frac{1}{q} = 1 - \frac{1}{p}$. So

\[
D=\left( \frac{4n^{2}+2n+\gamma }{2n(2n+2)},\frac{2n+(2n+1)\gamma }{2n(2n+2)}%
\right) =\left( \frac{1}{p_{D}},\frac{1}{q_{D}}\right) \text{ and\
}D^{\prime }=\left( 1-\frac{1}{q_{D}},1-\frac{1}{p_{D}}\right).
\]
\\ Since $0 \leq \gamma < 2n$ it is clear that $\| T_{\nu_{\gamma}}f \|_{p} \leq c \| f \|_{p}$ for all Borel function $f \in L^{p}(\mathbb{H}^{n})$ and all $1 \leq p \leq \infty$, so $(\frac{1}{p}, \frac{1}{p}) \in E_{\mu_{\gamma}}$.

Thus for $0 < \gamma < 2n$ the set $E_{\nu_{\gamma}}$ is contained in the closed trapezoid with vertices $(0,0)$, $(1,1)$, $D$ and $D^{\prime }$, and the set $E_{\nu_{0}}$ is contained in the closed triangle with vertices $(0,0)$, $(1,1)$, and $(\frac{2n+1}{2n+2}, \frac{1}{2n+2})$.

In the Section 3, our main result appears. There we prove that the operator $T_{\mu_A}$ is bounded from $L^{\frac{2n+2}{2n+1}}(\mathbb{H}^{n})$ to $L^{2n+2}(\mathbb{H}^{n})$, see Theorem 3 below. This result allows us to characterize the type set $E_{\nu_{0}}$ as well as the interior of $E_{\nu_{\gamma}}$ for $0 < \gamma < 2n$. More precisely, we show that $E_{\nu_0}$ is the closed triangle with vertices $(0,0)$, $(1,1)$, and $(\frac{2n+1}{2n+2}, \frac{1}{2n+2})$ and the interior of $E_{\nu_{\gamma}}$ coincides with the interior of the closed trapezoid with vertices 
$(0,0)$, $(1,1)$, $D$ and $D^{\prime }$, see Theorem 4 and Theorem 6 below.

Throughout this paper, $c$ will denote a positive real constant not necessarily the same at each occurrence. The symbol $A \lesssim B$ stands for the inequality $A \leq cB$ for some constant $c$. We use the following convention for the Fourier transform in $\mathbb{R}^{n}$ 
$\widehat{f}(\xi) = \int f(x) e^{-i \xi \cdot x} dx$. The Fourier transform $\widehat{u}$ of a distribution $u$ on $\mathbb{R}^{n}$ is the distribution defined by $(\widehat{u}, \phi) = (u, \widehat{\phi} \, )$ for all rapidly decreasing functions $\phi$ on $\mathbb{R}^{n}$.

\section{Preliminaries}

In the sequel $J$ will denote the $2n \times 2n$ skew-symmetric matrix defined in (\ref{matrix J}). It is easy to check that 

\

\textbf{a)} $J^{2}=-I$. 

\

\textbf{b)} $J^{t} = -J$.

\

\textbf{c)} $x^{t}Jx = 0$, for all $x \in \mathbb{R}^{2n}$. 

\

\textbf{d)} $x^{t}Jy = -y^{t}Jx$, for all $x,y \in \mathbb{R}^{2n}$.

\begin{lemma}
Let $A$ be a $2n \times 2n$ real diagonal matrix. Then
\[
\det(A \pm J) = (a_{11} \, a_{(n+1)(n+1)} +1) \cdot (a_{22} \, a_{(n+2)(n+2)} +1) \cdot \cdot \cdot (a_{nn} \, a_{(2n)(2n)} +1),
\]
where the $a_{ii}$'s are the diagonal entries of $A$.
\end{lemma}

\begin{proof}
Since $\det(A+J) = \det((A+J)^{t}) = \det(A-J)$, it is sufficient to prove the statement of the lemma for $\det(A+J)$.
Applying induction on $n$ the lemma follows.
\end{proof}

\begin{proposition}
Let $A$ be a $2n \times 2n$ real symmetric matrix. Then the graph of the function $\varphi(y) = y^{t}Ay$ generates all the group $\mathbb{H}^{n}$. Moreover, the measure $\nu_0 = \eta \mu_A$ is equivalent to the measure $\eta \sigma$, where $\eta$ is a cut-off function and $\sigma$ is the surface measure on the graph of $\varphi$.
\end{proposition}

\begin{proof}
The first statement will follow if we prove that $(x,0)$ and $(0,t)$ belong to the set $G_{\gr (\varphi)}$ generated by the graph $\gr (\varphi)$ of $\varphi$, since $(x,t) = (x,0) \cdot (0,t)$. It is clear that $(x, \varphi(x)) \in G_{\gr (\varphi)}$, so 
$(- t^{1/2}x, \varphi(t^{1/2}x)) = (-t^{1/2}x, \varphi(- t^{1/2}x)) \in G_{\gr (\varphi)}$ for all $x \in \mathbb{R}^{2n}$ and all $t >0$. From that it follows that $(0,t\varphi(x)) \in G_{\gr (\varphi)}$ for all $t >0$ and all $x$. If $A$ is a non-null matrix, then $(0,-t) = (0,t)^{-1} \in G_{\gr (\varphi)}$ and $(x, 0)= (x, \varphi(x)) \cdot (0, -\varphi(x)) \in G_{\gr (\varphi)}$. If $A$ is the null matrix, it is sufficient to prove that $(0,t) \in G_{\gr (\varphi)}$ for all $t$. Indeed, for $x$ and $y$ such that $\langle x , y \rangle \neq 0$ we have
$(0,t) = (x,0) \cdot (\frac{ty}{\langle x , y \rangle}, 0) \cdot (-x - \frac{ty}{\langle x , y \rangle}, 0) \in G_{\gr (\varphi)}$. 
So $G_{\gr (\varphi)} = \mathbb{H}^{n}$.  

For the second part of the proposition, we have that the surface measure on the graph of $\varphi$ is given by
\[
\sigma(E) = \int_{\phi^{-1}(E)} \sqrt{\det[(\partial_{x_i} \phi, \partial_{x_j} \phi)_{x}]} \, dx,
\]
where $\phi(x) =(x, \varphi(x))$ and $E$ is a borelian set of $\mathbb{R}^{2n+1}$ (see pp. 43-45 in \cite{Carmo}). A computation gives
\[
\det[(\partial_{x_i} \phi, \partial_{x_j} \phi)_{x}]= 1+ \sum_{j=1}^{2n} (\partial_{x_j} \varphi(x))^{2}, \,\, for \,\, all \,\, x.
\]
So
\[
\int_{\mathbb{R}^{2n}} \chi_{E}(\phi(x)) \eta(x) dx \leq \int_{\phi^{-1}(E)} \sqrt{\det[(\partial_{x_i} \phi, \partial_{x_j} \phi)_{x}]} \eta(x) \, dx \lesssim \int_{\mathbb{R}^{2n}} \chi_{E}(\phi(x)) \eta(x) dx. 
\]
Then $\nu_0$ is equivalent to $\eta \sigma$.
\end{proof}

The $\lambda$-twisted convolution is defined by

\[
(f \times_{\lambda} g) (x) = \int_{\mathbb{R}^{2n}} f(x-y) g(y) e^{-i \lambda  x^{t} J y } dy.
\]
\\ Given a $2n \times 2n$ real symmetric matrix $A$, we put

\[
e_{A}(x)= e^{ix^{t} Ax}.
\]
\\ It is easy to check, using the properties ${\bf b)}$ and ${\bf c)}$ of the matrix $J$, that

\[
(f \times_{\lambda} e_{\lambda A})(x)= e_{\lambda A}(x) (e_{\lambda A}(\cdot) f(\cdot)) \,\, \widehat{} \, \left(\lambda (2A+J)x \right),
\]
\\ where $\widehat{f}(\xi) = \int_{\mathbb{R}^{2n}} f(x) e^{-i x \cdot \xi} dx$ is the Fourier transform of $f$. 
Thus for each $f \in L^{1}(\mathbb{R}^{2n}) \cap L^{2}(\mathbb{R}^{2n})$ we have

\begin{equation} \label{producto twisted}
\left\Vert f \times_{\lambda} e_{\lambda A} \right\Vert_{L^{2}(\mathbb{R}^{2n})} = (2\pi)^{n} |\lambda|^{-n} |\det(2A\pm J) |^{-1/2} \left\Vert f \right\Vert_{L^{2}(\mathbb{R}^{2n})},
\end{equation}
\\ if  $\det(2A\pm J) \neq 0$.

\section{Main result}

To prove the $L^{\frac{2n+2}{2n+1}}(\mathbb{H}^{n})-L^{2n+2} (\mathbb{H}^{n})$ boundedness of the operator $T_{\mu_A}$ we embed our operator in an analytic family $\{ T_z \}$ of operators on the strip $-n \leq \Re(z) \leq 1$, and then we apply the complex interpolation theorem.  

\begin{theorem} If $\det(2A \pm J) \neq 0$, then the operator $T_{\mu_A}$ is bounded from $L^{\frac{2n+2}{2n+1}}(\mathbb{H}^{n})$ to  $L^{2n+2} (\mathbb{H}^{n})$. 
\end{theorem}

\begin{proof} To prove the statement of the theorem we consider the family $\{ |s|^{z-1} \}$ of functions initially defined when $\Re(z)>0$ and $s\in \mathbb{R} \setminus\{ 0 \}$. This family of functions can be extended, in the $z$ variable, to an analytic family of distributions on $\mathbb{C} \setminus \{ -2k : k \in \mathbb{N} \cup \{0 \} \}$. By abuse of notation, we denote this extension by $|s|^{z-1}$. The family 
$\{ |s|^{z-1} \}$ have simple poles in $z = -2k$ for $k \in \mathbb{N} \cup \{0 \}$. Since the meromorphic continuation of the function $\Gamma \left( \frac{z}{2}\right)$ (we keep the notation for his continuation) has simple poles at the same points (i.e. $z = -2k$), the family $\{ I_z \}$ of distributions defined by
\begin{equation}
I_{z}(s)=\frac{2^{-\frac{z}{2}}}{\Gamma \left( \frac{z}{2}\right) }   
| s | ^{z-1}  \label{iz}
\end{equation}
results in an entire family of distributions (see pp. 55-56 in \cite{Gelfand}).

From this construction, and by taking the ratios of the corresponding residues at $z=0$, we have $I_{0} = \delta$,  where $\delta $ is the Dirac distribution at the origin on $\mathbb{R}$ (see equation (3), pp. 57 in \cite{Gelfand}), also $\widehat{I_{z}}= cI_{1-z}$ for some real constant $c$ independent of $z$ (see equation ($12'$), pp. 173 in \cite{Gelfand}).

For $z \in \mathbb{C}$, we also define $U_{z}$ as the distribution on $\mathbb{H}^{n}$ given by the tensor product
$$U_{z} = \delta_{\mathbb{R}^{2n}} \otimes I_{z},$$ where $\delta_{\mathbb{R}^{2n}}$ is the Dirac
distribution at the origin on $\mathbb{R}^{2n}$ and $I_z$ is given by (\ref{iz}).

Let $\{ T_z \}$ be the analytic family of operators on the strip $-n \leq \Re(z) \leq 1$, given by
\[
T_z f = f \ast \mu_A \ast U_{z}.
\]
It is clear that $T_{0} =  T_{\mu_A}$. For $\Re(z)=1$ we have
\[
\left\Vert T_{z}f \right\Vert_{\infty} = \left\Vert f \ast \mu_A \ast U_{z} \right\Vert_{\infty} \leq
\left\Vert f \right\Vert_{1} \left\Vert \mu_A \ast U_{z} \right\Vert_{\infty}.
\]
Since $\mu_A \ast U_{1+ib} (x,t) = I_{1+ib} (t- \varphi(x)) =
\frac{2^{-(1+ib)/2}}{\Gamma((1+ib)/2)} |t- \varphi(x)|^{ib}$ it follows that
\[
\left\Vert T_{1+ib} \right\Vert_{1, \infty} \leq \left| \frac{2^{-(1+ib)/2}}{\Gamma((1+ib)/2)} \right|, \,\, \forall \, b \in \mathbb{R}.
\]
For $\Re(z)=-n$, we will prove that the operator $T_{z}$ is bounded on $L^{2}(\mathbb{H}^{n})$. This is equivalent to showing that 
\[
\int_{\mathbb{R}^{2n}} \, \left| (T_{z}f)^{\lambda}(x) \right|^{2} \, dx \leq c \int_{\mathbb{R}^{2n}} \, \left| f^{\lambda}(x) \right|^{2} \, dx,
\]
where $h^{\lambda}(x):= \int_{\mathbb{R}} h(x,t) e^{-i\lambda t} \, dt.$ A computation gives
\[
(T_{-n+ib}f)^{\lambda}(x) = \widehat{I_{-n+ib}}(\lambda) \int_{\mathbb{R}^{2n}} \, f^{\lambda}(x-y) e_{\lambda A}(y) 
e^{-i \lambda x^{t}J y} dy
\]
\[= \widehat{I_{-n+ib}}(\lambda) \, \left(f^{\lambda} \times_{\lambda} e_{\lambda A} \right)(x).
\]
From the identity in (\ref{producto twisted}) and since $\widehat{I_{z}}= c I_{1-z}$, we get
\[
\left\Vert (T_{-n+ib}f)^{\lambda} \right\Vert_{L^{2}(\mathbb{R}^{2n})} = \left| \frac{c2^{-(1+n-ib)/2}}{\Gamma \left(\frac{1+n-ib}{2} \right)} \right| (2\pi)^{n} |\det(2A \pm J)|^{-1/2} \Vert f^{\lambda} \Vert_{L^{2}(\mathbb{R}^{2n})},
\]
for each $b \in \mathbb{R}$. So $T_{-n+ib}$ is bounded on $L^{2}(\mathbb{H}^{n})$ if $\det(2A \pm J) \neq 0$. Finally, it is easy to see, with the aid of the Stirling formula (see e.g. \cite{stein4}), that the family $\left\{ T_{z}\right\} $ satisfies, 
on the strip $-n\leq \Re(z)\leq 1$, the hypothesis of the complex interpolation theorem (see \cite{stein2}, pp. 205) and 
so $T_{0}= T_{\mu_A}$ is bounded from $L^{\frac{2n+2}{2n+1}}(\mathbb{H}^{n})$ into $L^{2n+2}(\mathbb{H}^{n})$. 
\end{proof}

\begin{theorem} Let $\nu_0$ be the measure defined by $(\ref{nu})$ with $\gamma = 0$. If $\det(2A \pm J) \neq 0$, then the type set $E_{\nu_0}$ is the closed triangle with vertices $(0,0)$, $( 1,1)$, and $\left( \frac{2n+1}{2n+2},\frac{1}{2n+2}\right)$.
\end{theorem}

\begin{proof}
Since the following inequality
$$T_{\nu_0}f \leq T_{\mu_A}f$$ holds for each borelian function $f \geq 0$, the theorem follows from the restrictions that appear in (\ref{restriccion 1}), Theorem 3 and the Riesz convexity Theorem.
\end{proof}

\begin{corollary}
If $\det(2A \pm J) \neq 0$, then the operator $T_{\mu_A}$ is bounded from $L^{p}(\mathbb{H}^{n})$ into $L^{p}(\mathbb{H}^{n})$ if and only if 
$p= \frac{2n+2}{2n+1}$ and $q= 2n+2$.
\end{corollary}

\begin{proof}
The "\textit{if}" part of the corollary is the theorem 3. To see the reciprocal we introduce the action of the dilation group 
$\mathbb{R}^{>0}$ on $\mathbb{H}^{n}$, i.e.: $\delta \cdot (x,t) =(\delta x, \delta^{2}t)$, $\delta > 0$. For a function $f$ defined on $\mathbb{H}^{n}$ we put $f_{\delta}(x,t) = f(\delta \cdot (x,t))$. It is easy to check that
\[
(T_{\mu_A} f)_{\delta} = \delta^{2n} \, T_{\mu_A} (f_{\delta}).
\]
If $\| T_{\mu_A} f\|_q \leq c_{p,q} \|f \|_p$, then 
\[
\delta^{-\frac{2n+2}{q}} \| T_{\mu_A} f\|_q = \| (T_{\mu_A} f)_{\delta}\|_q = \delta^{2n} \| T_{\mu_A} (f_{\delta} )\|_q\leq \delta^{2n} c \|f_{\delta} \|_p = \delta^{2n - \frac{2n+2}{p}} c \| f \|_p,
\] 
for all $\delta >0$. So $\frac{1}{q} = \frac{1}{p} -\frac{2n}{2n+2}$. Since $T_{\nu_0}f \leq T_{\mu_A}f$ for $f \geq 0$. From Theorem 4 it follows that $p= \frac{2n+2}{2n+1}$ and $q= 2n+2$.
\end{proof}

We recall that
\[
D=\left( \frac{4n^{2}+2n+\gamma }{2n(2n+2)},\frac{2n+(2n+1)\gamma }{2n(2n+2)}%
\right) =\left( \frac{1}{p_{D}},\frac{1}{q_{D}}\right) \text{ and\
}D^{\prime }=\left( 1-\frac{1}{q_{D}},1-\frac{1}{p_{D}}\right).
\]

\begin{theorem}
Let $\nu_{\gamma}$ be the measure defined by $(\ref{nu})$ with $0 < \gamma < 2n$. If $\det(2A \pm J) \neq 0$, then the type set $E_{\nu_{\gamma}}$ is contained in the closed trapezoid with vertices $(0,0)$, $(1,1)$, $D$ and $D'$ with the only possible exception of the closed segment joining the two points
$D$ and $D'$.
\end{theorem}

\begin{proof}
For each $k \in \mathbb{N} \cup \{0\}$ we define the sets $A_{k} \subset \mathbb{R}^{2n}$ by
\[
A_{k}= \left\{ y \in \mathbb{R}^{2n} : 2^{-k} < \left\vert y \right\vert \leq 2^{-k+1} \right\}
\]
Let $\nu_{\gamma, k}$ be the fractional Borel measure given by
\[
\nu_{\gamma, k}(E)= \int_{A_{k}} \chi_{E} \left(y, \varphi(y) \right) \eta \left( y \right) | y |^{- \gamma} dy
\]
and let $T_{\nu_{\gamma, k}}$ be its corresponding convolution operator, i.e: $T_{\nu_{\gamma, k}}f=f \ast \nu_{\gamma, k}$.
Now, it is clear that  $\nu_{\gamma}=\sum_{k}\nu_{\gamma, k}$ and $\left\Vert T_{\nu_{\gamma}} \right\Vert_{p,q} \leq \sum_{k} \left\Vert T_{\nu_{\gamma, k}} \right\Vert_{p,q}$.
For $f\geq 0$ we have that
$$
\int_{\mathbb{H}^{n}} f(y,s) d\nu_{\gamma, k}(y,s) \leq 2^{k\gamma}\int_{\mathbb{R}^{2n}} f\left(y, \varphi(y) \right) \eta \left( y \right) dy.$$
Thus $\left\Vert T_{\nu_{\gamma, k}} \right\Vert_{p,q} \leq c 2^{k\gamma}\left\Vert T_{\nu_{0}} \right\Vert_{p,q}$, from Theorem 4 it follows that
\[
\left\Vert T_{\nu_{\gamma, k}} \right\Vert_{\frac{2n+2}{2n+1},2n+2} \leq c 2^{k\gamma}.
\]
It is easy to check that $\left\Vert T_{\nu_{\gamma, k}}
\right\Vert_{1,1}\leq \left\vert \nu_{\gamma, k} (\mathbb{R}^{2n+1})
\right\vert \sim \int_{A_{k}} \left\vert y \right\vert^{-\gamma} dy= c 2^{-k(2n-\gamma)}.$ \\
For $0< \theta <1$, we define
$$\left(\frac{1}{p_{\theta}}, \frac{1}{q_{\theta}} \right) = \left(\frac{2n+1}{2n+2}, \frac{1}{2n+2} \right) (1-\theta) + (1,1)\theta.$$
By the Riesz convexity Theorem we have
$$
\left\Vert T_{\nu_{\gamma, k}} \right\Vert_{p_{\theta},q_{\theta}} \leq c
2^{k\gamma(1-\theta)-k(2n-\gamma) \theta}.$$
Choosing $\theta$ such that $k\gamma(1-\theta)-k(2n-\gamma) \theta = 0$ yields $\displaystyle{\sup_{k \in \mathbb{N}}} \left\Vert T_{\nu_{\gamma, k}}
\right\Vert_{p_{\theta},q_{\theta}} \leq c < \infty$. A simple computation gives $\theta=\frac{2n-\gamma}{2n}$, then $\left( \frac{1}{p_{\theta}}, \frac{1}{q_{\theta}} \right) = \left( \frac{1}{p_{D}}, \frac{1}{q_{D}} \right)$, so $\left\Vert T_{\nu_{\gamma, k}} \right\Vert_{p_{D},q_{D}} \leq c$,
where $c$ is independent of $k$. Interpolating once again, but now between the points $\left(\frac{1}{p_{D}}, \frac{1}{q_{D}} \right)$ and $(1,1)$ we obtain, for each $0< \tau <1$ fixed
$$
\left\Vert T_{\nu_{\gamma, k}} \right\Vert_{p_{\tau},q_{\tau}} \leq c 2^{-k(2n-\gamma) \tau}.$$
Since $\left\Vert T_{\nu_{\gamma}} \right\Vert_{p,q} \leq \sum_{k} \left\Vert T_{\nu_{\gamma, k}} \right\Vert_{p,q}$ and $0< \gamma < 2n$, it follows that
\begin{equation*}
\left\Vert T_{\nu_{\gamma}} \right\Vert_{p_{\tau},q_{\tau}} \leq c\sum_{k \in \mathbb{N}} 2^{-k(2n-\gamma) \tau} <\infty.
\end{equation*}
By duality we also have
\begin{equation*}
\left\Vert T_{\nu_{\gamma}} \right\Vert_{\frac{q_{\tau}}{q_{\tau}-1},\frac{p_{\tau}}{p_{\tau}-1}}\leq c_{\tau} <\infty. 
\end{equation*}
Finally, the theorem follows from the Riesz convexity Theorem, and the restrictions that appear in (\ref{restriccion 1}) and (\ref{restriccion 2}).
\end{proof}

We conclude this note with the following remarks. 

\begin{remark}
Let $\nu_0$ be the measure of compact support defined by $(\ref{nu})$, but now with $\det(2A \pm J) = 0$. In this case, by Theorem 1.1 in \cite{ricci2} and Proposition 2, we can assure that the type set $E_{\nu_0}$ has a non-empty interior. 
\end{remark}

\begin{remark}
Lemma 1 provides us of examples of diagonal matrices $A$ such that $\det(2A \pm J)=0$. By the above remark we know that the interior of the type set of the measure $\nu_0  = \eta \mu_{A}$ is non-empty.  If $n \geq 2$ and $A$ also satisfies that 
$\varphi(y)=y^{t} Ay= \sum_{j=1}^{n} \alpha_j |y_j|^{2}$ 
$( \alpha_j \in \mathbb{R}$ and $y_j \in \mathbb{R}^{2})$, then the type set of $\nu_0$ is the closed triangle with vertices $(0,0)$, $(1,1)$ and 
$(\frac{2n+1}{2n+2}, \frac{1}{2n+2})$. This result is independent of the value of $\det(2A \pm J)$ $($see Theorem 1, pp. 102 in \cite{G-R}$)$.
\end{remark}

These final comments illustrate the limits of the techniques used in this note as well as of those developed in the works \cite{G-R} and \cite{G-R2}.

\

\textbf{Acknowledgement.} I express my thanks to Prof. Fulvio Ricci for his numerous useful suggestions. My thanks also go to the
referee for the useful suggestions and corrections which made the manuscript more readable.

\

\end{document}